\documentclass[11pt]{amsart}
\usepackage[leqno]{amsmath}
\usepackage{amsthm}
\usepackage{amsfonts}
\usepackage{amssymb}
\usepackage{eucal}
\usepackage{mathrsfs}
\usepackage[all]{xy}
\usepackage{pictexwd,dcpic} %commutative diagrams
\usepackage{enumerate}
\usepackage{longtable}
\usepackage{pdflscape}
\usepackage{hyperref}
\usepackage{graphicx}
\usepackage{xcolor}

\theoremstyle{plain}
\newtheorem{thm}{Theorem}[section]

\newtheorem{prop}[thm]{Proposition}
\newtheorem{lemma}[thm]{Lemma}

\theoremstyle{definition}

\newtheorem{rem}[thm]{Remark}
\newtheorem*{oldproof}{Proof}
\renewenvironment{proof}[1][{}]{\begin{oldproof}[#1]}{\qed\end{oldproof}}

\newcommand{\Q}{\mathbb{Q}}
\newcommand{\F}{\mathbb{F}}
\newcommand{\Or}{\mathcal{O}}
\newcommand{\Gal}{\text{\textnormal{Gal}}}
\newcommand{\Hom}{\text{\textnormal{Hom}}}
\newcommand{\Aut}{\text{\textnormal{Aut}}}
\newcommand{\ord}{\text{\textnormal{ord}}}
\newcommand{\loc}{\text{\textnormal{loc}}}
\newcommand{\Pic}{\text{\textnormal{Pic}}}
\newcommand{\mdiv}{\text{\textnormal{div}}}
\newcommand{\rk}{\text{\textnormal{rk}}}
\newcommand{\Ima}{\text{\textnormal{Im}}}

\newcommand{\mc}[1]{\ensuremath{\mathcal{#1}}}
\newcommand{\mb}[1]{\ensuremath{\mathbb{#1}}}
\newcommand{\ra}{\rightarrow}
\newcommand{\xra}{\xrightarrow}
\newcommand{\mono}{\hookrightarrow}
\newcommand{\mur}{\ensuremath{{\rm ur}}}
\newcommand{\mres}{\ensuremath{{\rm Res}}}
\newcommand{\mcor}{\ensuremath{{\rm Cor}}}
\newcommand{\msel}{\ensuremath{{\rm Sel}}}
\newcommand{\legendre}[2]{\left(\frac{#1}{#2}\right)}
\DeclareFontFamily{U}{wncy}{}
\DeclareFontShape{U}{wncy}{m}{n}{<->wncyr10}{}
\DeclareSymbolFont{mcy}{U}{wncy}{m}{n}
\DeclareMathSymbol{\Sha}{\mathord}{mcy}{"58}

\begin{document}

\title{On Conjectural Rank Parities of Quartic and Sextic Twists of Elliptic Curves}
\author[M. Weidner]{Matthew Weidner}
\address{Computer Laboratory, University of Cambridge, 15 JJ Thomson Avenue\\
Cambridge CB3 0FD, United Kingdom}
\email{malw2@cam.ac.uk}
\date{May 30, 2019}
\thanks{Electronic version of an article published as \textit{Int.\ J.\ Number Theory} \textbf{15}(9) (2019) 1895--1918.  \url{https://doi.org/10.1142/S1793042119501057}  \textcopyright World Scientific Publishing Company. \url{https://www.worldscientific.com/worldscinet/ijnt}}

\begin{abstract}
We study the behavior under twisting of the Selmer rank parities of a self-dual prime-degree isogeny on a principally polarized abelian variety defined over a number field, subject to compatibility relations between the twists and the isogeny.  In particular, we study isogenies on abelian varieties whose Selmer rank parities are related to the rank parities of elliptic curves with $j$-invariant 0 or 1728, assuming the Shafarevich-Tate conjecture.  Using these results, we show how to classify the conjectural rank parities of all quartic or sextic twists of an elliptic curve defined over a number field, after a finite calculation.  This generalizes previous results of Hadian and Weidner on the behavior of $p$-Selmer ranks under $p$-twists.
\end{abstract}

\maketitle

\section{Introduction}\label{s:i}
The study of ranks of elliptic curves is an old and difficult Diophantine problem.  Assuming the Shafarevich-Tate conjecture, it is well-known how to compute the rank of a given elliptic curve over a number field, but many questions remain about the properties of ranks in general.  Much recent work has centered on the behavior of ranks in families of elliptic curves, especially families of quadratic twists.  For instance, Swinnerton-Dyer \cite{swinnerton_dyer} determines the distribution of 2-Selmer ranks, which provide an upper bound on rank, in families of elliptic curves with full rational 2-torsion.  Later results by Klagsbrun \cite{klagsbrun_partial_2_torsion} and Klagsbrun, Mazur, and Rubin \cite{kbr_no_2_torsion} cover other types of rational 2-torsion.  Additionally, these same authors \cite{kbr} determine the distribution of 2-Selmer rank parities among quadratic twists of any elliptic curve, and they find that odd and even 2-Selmer ranks are evenly distributed whenever the base field has a real embedding.  Since 2-Selmer rank parities are the same as rank parities if we assume the Shafarevich-Tate conjecture, their work provides good evidence for Goldfeld's conjecture \cite[Conjecture B]{goldfeld}, which says that the average rank in any family of quadratic twists over $\Q$ is $1/2$.

Another result along these lines is the following.  Let $K$ be a number field, and let $E$ be an elliptic curve defined over $K$.  In \cite{quadratic_twists}, building off of results in \cite{kbr, mazur_towers, mazur_rubin, poonen_rains, yu} and other papers, it is shown how to classify the 2-Selmer rank parities of all quadratic twists of $E$ as a function of the twisting parameter, after a finite calculation.  Assuming the Shafarevich-Tate conjecture, this classifies the rank parities as well.  Specifically, for any $d \in K^*$, let $E_d$ be the quadratic twist of $E$ by $d$, and let $d_2(E_d)$ denote the 2-Selmer rank of $E_d$.  Also let $\Sigma$ be a finite set of places of $K$ containing all infinite places, all places dividing 2, and all places of bad reduction for $E$.  For each place $v$ of $K$, let $K_v$ be the localization of $K$ at $v$, let $\mathcal{O}_v$ be the ring of integers in $K_v$, let $\pi_v$ be a (fixed) uniformizer for $v$, and let $\ord_v$ be the $v$-adic valuation.  Then we have the following result.

\begin{thm}[{\cite[Theorem 4.6 and Remark 4.9]{quadratic_twists}}]
Let $c, d \in K^*$ be such that for all places $v \in \Sigma$, $c \equiv d \pmod{(K_v^*)^2}$.  Equivalently: \begin{itemize}
  \item Define the integers $m_v$ for all finite places $v \in \Sigma$ by
  $$m_v := \begin{cases} 2e_{v/2} + 1 & \mbox{if $v | 2$} \\ 1 & \mbox{else,} \end{cases}$$
  where $e_{v/2}$ is the ramification index of $v$ over 2;
  \item Assume that for all finite places $v \in \Sigma$,
  \begin{itemize}
    \item $\ord_v(c) \equiv \ord_v(d) \pmod{2}$
    \item $c/(\pi_v^{\ord_v(c)})$ and $d/(\pi_v^{\ord_v(d)})$ have the same residues in $(\Or_v/v^{m_v})^*/((\Or_v/v^{m_v})^*)^2$;
  \end{itemize}
  \item And assume that $c$ and $d$ have the same signs at all real embeddings.
\end{itemize}
Then $d_2(E_c) \equiv d_2(E_d) \pmod{2}$.
\end{thm}

The conditions of this theorem partition $K^*$ into finitely many classes, so by computing $d_2(E_d)$ for one $d$ in each class (it is well-known that the 2-Selmer rank is computable), we can classify the 2-Selmer rank parity of all quadratic twists of $E$.

Let $\bar{K}$ denote an algebraic closure of $K$.  When $E$ has $j$-invariant not equal to 1728 or 0, the quadratic twists of $E$ are precisely the elliptic curves defined over $K$ which are isomorphic to $E$ as curves over $\bar{K}$.  Hence the above result classifies the conjectural rank parities of all elliptic curves over $K$ which are $\bar{K}$-isomorphic to $E$.

However, when $E$ has $j$-invariant 1728 or 0, it is well-known that the $\bar{K}$-isomorphic curves are the quartic or sextic twists of $E$, respectively, not just the quadratic twists (see \cite[Proposition X.5.4]{silverman}).  In this paper, we find versions of the above result for quartic twists of $y^2 = x^3 + x$ over $K$ (Theorem \ref{main_thm} and Proposition \ref{mod_4_prop}) and sextic twists of $y^2 = x^3 + 1$ over $K$ (Theorem \ref{main_thm2} and Proposition \ref{mod_6_prop}).  By doing so, we complete the following task: Classify the (conjectural) rank parities of all elliptic curves over $K$ which are $\bar{K}$-isomorphic to a given curve, after performing a finite amount of computation.

Section \ref{background} provides background on the variation of Selmer rank parities under twists of abelian varieties, based on \cite{quadratic_twists, kbr, mazur_towers, mazur_rubin, poonen_rains, yu}.  The section culminates in Theorem \ref{main_bg}, which is the main result used in the sequel.  \S \ref{setup} sets up the general situation, \S \ref{metabolic}-\ref{sec_parity_twist_formula} explain existing results on how to write the variation in Selmer rank parity as a sum of local invariants (Theorem \ref{parity_twist_formula}), \S \ref{local_conditions} shows how to evaluate the local invariants in certain cases, and \S \ref{mu_n} specializes to the case that we will use in the sequel.  Sections \ref{quartic} and \ref{sextic} then show how to classify the conjectural rank parities of quartic and sextic twists of elliptic curves, respectively, after a finite calculation.  Finally, in Section \ref{example_q}, we give an explicit example, in which we determine the conjectural rank parities of all quartic twists of $y^2 = x^3 + x$ and sextic twists of $y^2 = x^3 + 1$ over $\Q$.
\smallskip

\section{Selmer Rank Parities for Twists of Abelian Varieties}\label{background}
In this section, we summarize and slightly generalize existing results on Selmer rank parities of abelian varieties, culminating in Theorem \ref{main_bg}.  This is based on material from several references including \cite{kbr, mazur_towers, mazur_rubin, poonen_rains, yu}, and it mostly follows the quadratic twist case (see \cite[\S 3]{quadratic_twists}).

\subsection{Notation}\label{setup}
Let $K$ be a number field.  Fix an algebraic closure $\bar{K}$ of $K$, and let $G_K := \Gal(\bar{K}/K)$ be the absolute Galois group of $K$.  For each place $v$ of $K$, fix an embedding of $\bar{K}$ into the algebraic closure $\bar{K}_v$ of the completion $K_v$ of $K$ at $v$. This gives an embedding of the absolute Galois group $G_{K_v} := \Gal(\bar{K}_v/K_v)$ into $G_K$. If $v$ is a finite place, then $\mc{O}_v$ denotes the ring of integers of $K_v$, $\pi_v$ a fixed uniformizer for $K_v$, $\ord_v(-)$ the $v$-adic valuation, and $k_v := \mc{O}_v / (\pi_v)$ the residue field of $K_v$.  Also, for any $G_{K_v}$-module $M$ such that the inertia subgroup of $G_{K_v}$ acts trivially on $M$, $H^1_\mur(G_{K_v}, M)$ denotes the subgroup of $H^1(G_{K_v}, M)$ consisting of unramified cocyles, i.e., the image of $H^1(G_{k_v}, M)$ under the inflation map.

Let $A$ be a principally polarized abelian variety over $K$ (i.e., $A$ is defined over $K$ and we have fixed a principal polarization of $A$ over $K$).  Let $\lambda: A \ra A$ be a self-dual isogeny of prime degree $p$, with kernel denoted by $A[\lambda]$.  Let $\Sigma$ denote a fixed finite set of places of $K$ containing all places of bad reduction for $A$, all places above $p$, and all archimedean places.

Finally, let $H$ be a $G_K$-stable subgroup of the automorphism group $\mbox{Aut}(A)$ of $A$ such that every $\eta \in H$ acts trivially on $A[\lambda]$ and commutes with $\lambda$ (i.e., $\eta \circ \lambda = \lambda \circ \eta$).  For any field $F$ containing $K$, we let $1_F$ denote the identity cocycle in $H^1(G_F, H)$.

\subsection{Global metabolic structures and Selmer structures}\label{metabolic}

We wish to construct a global metabolic structure on the $G_K$-module $A[\lambda]$, as defined in \cite{kbr}.  To state the definition, first let $V$ be a finite-dimensional vector space over $\F_p$, equipped with a quadratic form $q: V \ra \F_p$.  Let
\[
(v, w)_q := q(v+w) - q(v) - q(w)
\]
be the bilinear form associated to $q$.  A subspace $X \subset V$ is called a Lagrangian subspace if $q|_X = 0$ and $X^\perp = X$ with respect to $(,)_q$.  The quadratic space $(V, q)$ is called a metabolic space if $V$ has a Lagrangian subspace and $(,)_q$ is nondegenerate.

  Next, let $e_p : A[p] \times A[p] \ra \mu_p$ be the nondegenerate alternating Weil pairing of $A$, where $A[p]$ denotes the kernel of the multiplication-by-$p$ map. Since $\lambda$ has degree $p$, $A[\lambda] \subset A[p]$, so we get a map $e_p: A[\lambda] \times A[\lambda] \ra \mu_p$ by restriction.  Then composing this map with the cup product gives the local Tate pairing
\[
\langle \:,\: \rangle_v: H^1(G_{K_v}, A[\lambda]) \times H^1(G_{K_v}, A[\lambda]) \ra H^2(G_{K_v}, \mu_p) \mono \F_p
\]
for each place $v$ of $K$.  Note that the local Tate pairing is antisymmetric, and by local Tate duality, it is nondegenerate.

A global metabolic structure on $A[\lambda]$ consists of a quadratic form $q_v$ on $H^1(G_{K_v}, A[\lambda])$ for every place $v$ of $K$, such that:
\begin{enumerate}
\item{the quadratic space $(H^1(G_{K_v}, A[\lambda]), q_v)$ is a metabolic space for every $v$;}
\item{for every $v \notin \Sigma$, $H^1_\mur(G_{K_v}, A[\lambda])$ is an isotropic subspace with respect to $q_v$;}
\item{for any $c \in H^1(G_K, A[\lambda])$, $\sum_v q_v(\loc_v(c)) = 0$;}
\item{the bilinear form induced by $q_v$ is the local Tate pairing $\langle , \rangle_v$ for every $v$.}
\end{enumerate}

In many cases, we can construct a canonical global metabolic structure on $A[\lambda]$.  Specifically, assume that we are in the situation of \cite[\S 4]{poonen_rains}, i.e.,
\begin{equation}\tag{A1}\label{poonen_rains_assumption}
\left[\begin{tabular}{l}
\parbox{0.7\textwidth}{
\begin{enumerate}
  \item $p$ is odd; or
  \item $p = 2$ and $\lambda$ is of the form $\phi_{\mathcal{L}}$ for some symmetric line sheaf $\mathcal{L}$, where $\phi_{\mathcal{L}}$ is as defined in \cite[p.\ 60 and Corollary 5 on p.\ 131]{mumford}.
\end{enumerate}
}
\end{tabular}\right]
\end{equation}
When $p$ is odd, there is a unique global metabolic structure on $A[\lambda]$, defined by
\[
q_v(x) := \frac{1}{2}\langle x, x \rangle_v;
\]
see \cite[Lemma 3.4]{kbr}.  When instead $p = 2$, there is a canonical quadratic form $q_v: H^1(G_{K_v}, A[\lambda]) \ra \mb{F}_2$ arising from the Heisenberg group for each place $v$, as described in \cite[\S 4]{poonen_rains}.  These quadratic forms again give a global metabolic structure; see \cite[Propositions 4.7, 4.9, and 4.11]{poonen_rains}.

A Selmer structure \mc{S} for $A[\lambda]$ consists of an $\mb{F}_p$-subspace $H^1_\mc{S}(G_{K_v}, A[\lambda]) \subset H^1(G_{K_v}, A[\lambda])$ for every place $v$ of $K$, such that $H^1_\mc{S}(G_{K_v}, A[\lambda]) = H^1_\mur(G_{K_v}, A[\lambda])$ for all but finitely many $v$. We say that \mc{S} is Lagrangian if for every $v$, $H^1_\mc{S}(G_{K_v}, A[\lambda])$ is a Lagrangian subspace of $H^1(G_{K_v}, A[\lambda])$ with respect to the canonical global metabolic structure defined above. The Selmer group associated to \mc{S} is defined by
$$H^1_\mc{S}(G_K, A[\lambda]) := \ker \left( H^1(G_K, A[\lambda]) \ra \bigoplus_v \left(H^1(G_{K_v}, A[\lambda])/H^1_\mc{S}(G_{K_v}, A[\lambda])\right) \right),$$
where the sum runs over all places $v$ of $K$. Note that the Selmer group is finite, since it consists of cocycles which are unramified outside of a fixed finite set of places.

The main result we need concerning global metabolic structures and Selmer structures is the following.

\begin{thm}[{\cite[Theorem 3.9]{kbr}}] \label{t:ss}
Assume $\lambda$ satisfies condition (\ref{poonen_rains_assumption}).  Suppose \mc{S} and $\mc{S}'$ are Lagrangian Selmer structures for $A[\lambda]$. Then:
\begin{equation*}
\dim_{\mb{F}_p}(H^1_\mc{S}(G_K, A[\lambda])) - \dim_{\mb{F}_p}(H^1_{\mc{S}'}(G_K, A[\lambda]))
\end{equation*}
$$\equiv \sum_v \dim_{\mb{F}_p}\left(H^1_\mc{S}(G_{K_v}, A[\lambda])/\left(H^1_\mc{S}(G_{K_v}, A[\lambda]) \cap H^1_{\mc{S}'}(G_{K_v}, A[\lambda])\right)\right) \pmod{2},$$
where the sum is taken over all places $v$ of $K$.
\end{thm}

\subsection{Twisting}\label{twisting}
It is well-known 
%(see, e.g., \cite[Proposition 5.4]{silverman})
that the twists of $A$ over $K$ (i.e., abelian varieties defined over $K$ which are isomorphic to $A$ as abelian varieties over $\bar{K}$) are classified by $H^1(G_K, \mbox{Aut}(A))$, so $H^1(G_K, H)$ classifies a certain subset of those twists.  (If the natural map $H^1(G_K, H) \ra H^1(G_K, \mbox{Aut}(A))$ is not injective, this classification will have some redundancy.)  We will denote the twist of $A$ over $K$ corresponding to a cocycle $\chi \in H^1(G_K, H)$ by $A_{\chi}$.

By definition, for each $\chi$, there is an isomorphism $\psi: A_{\chi} \ra A$ defined over $\bar{K}$ such that $\chi(\sigma) = (\sigma . \psi) \circ \psi^{-1}$ for all $\sigma \in G_K$.  From the isogeny $\lambda: A \ra A$, we get an isogeny $\psi^{-1} \circ \lambda \circ \psi: A_{\chi} \ra A_{\chi}$ defined over $\bar{K}$; we will also denote this isogeny by $\lambda$.  This isogeny is in fact defined over $K$: if $\sigma \in G_K$, then
\begin{align*}
\sigma . (\psi^{-1} \circ \lambda \circ \psi)
&= (\sigma . \psi)^{-1} \circ (\sigma . \lambda) \circ (\sigma . \psi) \\
&= (\chi(\sigma) \circ \psi)^{-1} \circ \lambda \circ (\chi(\sigma) \circ \psi) \\
&= \psi^{-1} \circ \left(\chi(\sigma)^{-1} \circ \lambda \circ \chi(\sigma)\right) \circ \psi \\
&= \psi^{-1} \circ \lambda \circ \psi
\end{align*}
by the assumption that every $\eta \in H$ commutes with $\lambda$.  Note that if $\lambda: A \ra A$ satisfies condition (\ref{poonen_rains_assumption}), then so does $\lambda: A_\chi \ra A_\chi$ for each $\chi$, so we also get a canonical global metabolic structure on $A_\chi[\lambda]$.

Also note that $\chi(\sigma) |_{A[\lambda]}$ is the identity by assumption, so $((\sigma . \psi) \circ \psi^{-1})|_{A[\lambda]} = 1$, i.e., $\psi|_{A_\chi[\lambda]} = (\sigma . \psi)|_{A_\chi[\lambda]}$.  Hence $\psi|_{A_\chi[\lambda]}$ is defined over $K$, so that $\psi|_{A_\chi[\lambda]}: A_\chi[\lambda] \ra A[\lambda]$ is an isomorphism of $G_K$-modules.%  In particular, because $A[\lambda]$ is defined over $K$, so is $A_{\chi}[\lambda]$.

\subsection{The parity twist formula}\label{sec_parity_twist_formula}

To each isogeny $\lambda: A_\chi \ra A_\chi$, we can associate a Selmer group $\msel^{\lambda}(A/K)$ in the usual way.  It is well-known that the Selmer group is a finite-dimensional $\mb{F}_p$-vector space, since it consists of cocycles that are unramified outside of a fixed finite set of places.  We call its dimension
\[
d_\lambda(A_\chi) := \dim_{\F_p}(\msel^\lambda(A_\chi))
\]
the $\lambda$-Selmer rank of $A$.  Our goal in this subsection is to write the parity of $d_\lambda(A_\chi) - d_\lambda(A)$ as a sum of local invariants depending on $\chi$.

As mentioned above, we have an isomorphism of $G_K$-modules $\psi|_{A_\chi[\lambda]}: A_\chi[\lambda] \xra{\sim} A[\lambda]$.  This isomorphism induces a natural isomorphism
\[
H^1(G_{K_v}, A_{\chi}[\lambda]) \xra{\sim} H^1(G_{K_v}, A[\lambda]),
\]
for each place $v$ of $K$, which identifies the local Tate pairings (because $A[\lambda] \cong A_\chi[\lambda]$ identifies the Weil pairings).  When $p > 2$, this implies that it also identifies the unique global metabolic structures.  However, when $p = 2$, it appears that there is no \textit{a priori} reason for this isomorphism to identify the canonical global metabolic structures, so we must assume it:
\begin{equation}\tag{A2}\label{yu_assumption}
\left[\begin{tabular}{l}
\parbox{0.7\textwidth}{
If $p = 2$, then the natural isomorphism
\[
H^1(G_{K_v}, A_{\chi}[\lambda]) \xra{\sim} H^1(G_{K_v}, A[\lambda])
\]
identifies the canonical quadratic forms constructed in \cite[\S 4]{poonen_rains}, for all places $v$ of $K$ and cocycles $\chi \in H^1(G_K, H)$.
}
\end{tabular}\right]
\end{equation}

We can now associate a Selmer structure $\mc{S}_\chi$ to any cocycle $\chi \in H^1(G_K, H)$: for each place $v$, let $H^1_{\mc{S}_\chi}(G_{K_v}, A[\lambda]) \subset H^1(G_{K_v}, A[\lambda])$ be the image under the Kummer map of $A_\chi(K_v) / \lambda A_\chi(K_v)$, where we have used the natural isomorphism to identify $H^1(G_{K_v}, A_\chi[\lambda])$ with $H^1(G_{K_v}, A[\lambda])$. It is well-known that this construction does indeed define a Selmer structure.  This Selmer structure is Lagrangian when $\chi = 1_K$ by \cite[Propositions 4.9 and 4.11]{poonen_rains}, and it is Lagrangian for general $\chi$ because the natural isomorphism identifies the local Tate pairings and canonical global metabolic structures.

For any place $v$ of $K$ and any cocycle $\chi \in H^1(G_K, H)$, define the local invariant $\delta_v(A, \chi)$ by
$$\delta_v(A, \chi) := \dim_{\mb{F}_p} \left( H^1_{\mc{S}_{1_K}}(G_{K_v}, A[\lambda])/(H^1_{\mc{S}_{1_K}}(G_{K_v}, A[\lambda]) \cap H^1_{\mc{S}_\chi}(G_{K_v}, A[\lambda])) \right).$$
Then applying Theorem \ref{t:ss} to our situation gives the following theorem, which I call the ``parity twist formula'':

\begin{thm}\label{parity_twist_formula}
Assume $\lambda$ satisfies conditions (\ref{poonen_rains_assumption}) and (\ref{yu_assumption}).  Then for any cocycle $\chi \in H^1(G_K, H)$, we have
$$d_\lambda(A) - d_\lambda(A_\chi) \equiv \sum_v \delta_v(A, \chi) \pmod{2},$$
where the sum is taken over all places $v$ of $K$.
\end{thm}

\subsection{Local conditions}\label{local_conditions}
Fix a cocycle $\chi \in H^1(G_K, H)$ and a place $v$ of $K$.  We wish to determine $\delta_v(A, \chi)$ in certain cases.  The results in this subsection are adapted from \cite[Lemmas 2.9-11 and 2.15-16]{yu}.

Let $F_w$ be a finite extension of $K_v$ such that the restriction of $\chi$ to $F_w$ is trivial, hence $A \cong A_\chi$ over $F_w$.

\begin{lemma}
\label{unramified}
Assume $A$ has good reduction at $v$ and $F_w/K_v$ is unramified.  Then $\delta_v(A, \chi) = 0$.
\end{lemma}
\begin{proof}
Let $N: A(F_w) \ra A(K_v)$ denote the norm map.  By the same argument as in \cite[Lemma 2.9]{yu}, $H^1_{\mc{S}_{1_K}}(G_{K_v}, A[\lambda]) \cap H^1_{\mc{S}_\chi}(G_{K_v}, A[\lambda])$ contains the image of $\left(N(A(F_w)) + \lambda(A(K_v))\right)/\lambda(A(K_v))$ under the Kummer map, as follows.  Consider the commutative diagram
\begin{align}\label{res_cor}
\begindc{\commdiag}[100]
\obj(0, 4)[topL]{$H^1(G_{F_w}, A[\lambda])$}
\obj(0, 0)[bottomL]{$H^1(G_{K_v}, A[\lambda])$}
\obj(10, 4)[topC]{$H^1(G_{F_w}, A[\lambda])$}
\obj(10, 0)[bottomC]{$H^1(G_{K_v}, A[\lambda])$}
\obj(20, 4)[topR]{$H^2(G_{F_w}, \mu_p) \mono \F_p$}
\obj(20, 0)[bottomR]{$H^2(G_{K_v}, \mu_p) \mono \F_p$.}
\mor{topL}{bottomL}{$\mcor$}
\mor{bottomC}{topC}{$\mres$}
\mor{topR}{bottomR}{$\mcor = \text{id}$}
\mor{topC}{topR}{$\cup$}
\mor{bottomC}{bottomR}{$\cup$}
\obj(5, 4)[topX]{$\times$}
\obj(5, 0)[bottomX]{$\times$}
\enddc
\end{align}
By \cite[Proposition 4.11]{poonen_rains}, $H^1_{\mc{S}_{1_F}}(G_{F_w}, A[\lambda])$ is its own orthogonal complement in $H^1(G_{F_w}, A[\lambda])$, so
\begin{align*}
H^1_{\mc{S}_{1_F}}(G_{F_w}, A[\lambda]) &= H^1_{\mc{S}_{1_F}}(G_{F_w}, A[\lambda])^\perp \\ &\subset \mres(H^1_{\mc{S}_{1_K}}(G_{K_v}, A[\lambda]))^\perp \\ &= \mcor^{-1}(H^1_{\mc{S}_{1_K}}(G_{K_v}, A[\lambda])).
\end{align*}
Similarly,
\[
H^1_{\mc{S}_{\chi}}(G_{F_w}, A[\lambda]) \subset \mcor^{-1}(H^1_{\mc{S}_{\chi}}(G_{K_v}, A[\lambda])).
\]
But $\mcor^{-1}(H^1_{\mc{S}_{1_K}}(G_{F_w}, A[\lambda])) = \mcor^{-1}(H^1_{\mc{S}_{\chi}}(G_{F_w}, A[\lambda]))$ since $A \cong A_\chi$ over $F_w$.  Then letting $i$ denote the Kummer map,
\begin{align*}
i((N(A(F_w)) + \lambda(A(K_v)))/\lambda(A(K_v))) &= \mcor(H^1_{\mc{S}_{\chi}}(G_{F_w}, A[\lambda])) \\ &\subset H^1_{\mc{S}_{1_K}}(G_{K_v}, A[\lambda]) \cap H^1_{\mc{S}_\chi}(G_{K_v}, A[\lambda])
\end{align*}
since the restriction of $\mcor: H^1(G_{F_w}, A[\lambda]) \ra H^1(G_{K_v}, A[\lambda])$ to $i(A(F_w)/\lambda(A(F_w)))$ induces the norm map.

Now by assumption, $F_w/K_v$ is unramified, so by \cite[Corollary 4.4]{mazur_towers}, $N(A(F_w)) = A(K_v)$.  Hence
\[
H^1_{\mc{S}_{1_K}}(G_{K_v}, A[\lambda]) = i(A(K_v)/\lambda(A(K_v))) \subset H^1_{\mc{S}_{1_K}}(G_{K_v}, A[\lambda]) \cap H^1_{\mc{S}_\chi}(G_{K_v}, A[\lambda]),
\]
so that $\delta_v(A, \chi) = 0$ by definition.
\end{proof}

\begin{lemma}
\label{pro_p_part}
Assume $v \nmid p\infty$.  Then
\[
A_\chi(K_v)/\lambda(A_\chi(K_v)) \cong A_\chi(K_v)[p^\infty]/\lambda(A_\chi(K_v)[p^\infty]).
\]
\end{lemma}
\begin{proof}
Easily $\lambda$ is surjective on the pro-(prime to $p$) part of $A_\chi(K_v)$, so only the pro-$p$ part $A_\chi(K_v)[p^\infty]$ contributes to $A_\chi(K_v)/\lambda(A_\chi(K_v))$.
\end{proof}

\begin{lemma}
\label{size_and_ramified}
Assume $v \nmid p\infty$.  Then:
\begin{enumerate}
  \item $\dim_{\F_p} \left( H^1_{\mc{S}_\chi}(G_{K_v}, A[\lambda]) \right) = \dim_{\F_p} \left(A_\chi(K_v)/\lambda(A_\chi(K_v))\right) = \dim_{\F_p} (A_\chi(K_v)[\lambda])$.
  \item If $A$ has good reduction at $v$ and $F_w/K_v$ is totally ramified, then $A_\chi(K_v)/\lambda(A_\chi(K_v)) \cong A_\chi(F_w)/\lambda(A_\chi(F_w))$.
\end{enumerate}
\end{lemma}
\begin{proof}
Since $A(K_v)[\lambda] \subset A(K_v)[p^\infty]$, we have an exact sequence
\[
0 \ra A(K_v)[\lambda] \ra A(K_v)[p^\infty] \xra{\lambda} A(K_v)[p^\infty] \ra A(K_v)[p^\infty]/\lambda(A(K_v)[p^\infty]) \ra 0.
\]
Because $A(K_v)[p^\infty]$ is finite, (1) then follows from Lemma \ref{pro_p_part}.

To prove (2), note that by our assumptions, $K_v(A[p^\infty])/K_v$ is an unramified extension.  Since $F_w/K_v$ is a totally ramified extension, $K_v(A(F_w)[p^\infty]) = K_v$.  Then $A(F_w)[p^\infty] = A(K_v)[p^\infty]$, so
\[
A_\chi(F_w)[p^\infty]/\lambda(A_\chi(F_w)[p^\infty]) = A_\chi(K_v)[p^\infty]/\lambda(A_\chi(K_v)[p^\infty])
\]
and the conclusion follows by Lemma \ref{pro_p_part}.
\end{proof}

\begin{lemma}
\label{totally_ramified}
Assume $H \cong \mu_{p^n}$ as a $G_K$-module for some $n \ge 1$ and $\ker(1 - \eta) = \ker(\lambda)$ for some generator $\eta$ of $H$.  Also assume $v \nmid p\infty$, $A$ has good reduction at $v$, $F_w = \bar{K}_v^{\ker(\chi)}$, and $F_w/K_v$ is a totally ramified degree $p^n$ extension.  Then
\[
H^1_{\mc{S}_{1_K}}(G_{K_v}, A[\lambda]) \cap H^1_{\mc{S}_\chi}(G_{K_v}, A[\lambda]) = 0,
\]
hence $\delta_v(A, \chi) \equiv \dim_{\mb{F}_p}(A(K_v)[\lambda]) \pmod{2}$.
\end{lemma}
\begin{proof}
The last statement follows from the definition of $\delta_v(A, \chi)$ and Lemma \ref{size_and_ramified}(1).

To prove the intersection, first note that the norm map $N: A(F_w)/\lambda(A(F_w)) \ra A(K_v)/\lambda(A(K_v))$ is 0.  Indeed, let $a \in A(F_w)$.  By Lemma \ref{size_and_ramified}(2), we can write $a = \lambda(b) + c$ with $b \in A(F_w)$ and $c \in A(K_v)$.  Then $N(a) = \lambda(N(b)) + [F_w: K_v]c \in \lambda(A(K_v))$ because $p \mid [F_w: K_v]$.  Because $\mcor: H^1(G_{F_w}, A[\lambda]) \ra H^1(G_{K_v}, A[\lambda])$ induces the norm map when restricted to the image of the Kummer map, this implies that $\mcor(H^1_{\mc{S}_{1_F}}(G_{F_w}, A[\lambda])) = 0$.

Now taking orthogonal complements,
\[
H^1_{\mc{S}_{1_K}}(G_{K_v}, A[\lambda]) \cap H^1_{\mc{S}_\chi}(G_{K_v}, A[\lambda]) = 0
\]
is equivalent to
\[
H^1_{\mc{S}_{1_K}}(G_{K_v}, A[\lambda]) + H^1_{\mc{S}_\chi}(G_{K_v}, A[\lambda]) = H^1(G_{K_v}, A[\lambda]).
\]
Thus it suffices to show
\[
H^1(G_{K_v}, A[\lambda]) \subset H^1_{\mc{S}_{1_K}}(G_{K_v}, A[\lambda]) + H^1_{\mc{S}_\chi}(G_{K_v}, A[\lambda]).
\]

So, let $x \in H^1(G_{K_v}, A[\lambda])$.  By (\ref{res_cor}) and properties of the cup product, along with the previous paragraph, we have
\[
\mres^{-1}(H^1_{\mc{S}_{1_F}}(G_{F_w}, A[\lambda])) = \left(\mcor(H^1_{\mc{S}_{1_F}}(G_{F_w}, A[\lambda]))\right)^\perp = 0^\perp = H^1(G_{K_v}, A[\lambda]).
\]
Thus $\mres(x) \in H^1_{\mc{S}_{1_F}}(G_{F_w}, A[\lambda]) = \mres\left(H^1_{\mc{S}_{1_K}}(G_{K_v}, A[\lambda])\right)$, so there is a $y \in H^1_{\mc{S}_{1_K}}(G_{K_v}, A[\lambda])$ such that $x - y \in \ker(\mres)$.  Then to prove the claim, it suffices to show
\[
\ker(\mres) \subset H^1_{\mc{S}_{1_K}}(G_{K_v}, A[\lambda]) + H^1_{\mc{S}_\chi}(G_{K_v}, A[\lambda]).
\]

Let $\psi: A_\chi \xra{\sim} A$ be an isomorphism such that $\sigma . \psi = \chi(\sigma) \circ \psi$ for $\sigma \in G_K$, necessarily defined over $F_w$.  Let $i$ and $i_\chi$ denote the Kummer maps for $A(K_v)/\lambda(A(K_v))$ and $A_\chi(K_v)/\lambda(A_\chi(K_v))$, respectively, with images in $H^1(G_{K_v}, A[\lambda])$.  Define a map
\[
\phi: A(K_v)[\lambda] \ra \ker(\mres)
\]
\[
\phi(P) := i(P) - i_\chi(\psi^{-1}(P)).
\]
Note that $\phi$ does indeed have image contained in $\ker(\mres)$, since if $P \in A(K_v)[\lambda]$, then $\mres(i(P)) = \mres(i_\chi(\psi^{-1}(P)))$ by the commutative diagram
\[
\begindc{\commdiag}[10]
\obj(-100, 0)[topL]{$A(K_v)/\lambda(A(K_v))$}
\obj(0, 0)[topC]{$H^1(G_{K_v}, A[\lambda])$}
\obj(100, 0)[topR]{$A_\chi(K_v)/\lambda(A_\chi(K_v))$}
\obj(-100, -40)[bottomL]{$A(F_w)/\lambda(A(F_w))$}
\obj(0, -40)[bottomC]{$H^1(G_{F_w}, A[\lambda])$}
\obj(100, -40)[bottomR]{$A_\chi(F_w)/\lambda(A_\chi(F_w))$}
\mor{topC}{bottomC}{$\mres$}
\mor{topL}{bottomL}{}
\mor{topR}{bottomR}{}
\mor{topL}{topC}{$i$}[\atleft, \injectionarrow]
\mor{bottomL}{bottomC}{}[\atleft, \injectionarrow]
\mor{topR}{topC}{$i_\chi$}[\atright, \injectionarrow]
\mor{bottomR}{bottomC}{}[\atright, \injectionarrow]
\cmor((100,-50)(95,-60)(80,-65)(0,-65)(-80,-65)(-95,-60)(-100,-50))
  \pup(0, -59){$\psi$}
\cmor((100,-50)(95,-60)(80,-65)(0,-65)(-80,-65)(-95,-60)(-100,-50))
  \pup(0, -71){$\sim$}
\enddc
\]

Since
\[
\Ima(\phi) \subset H^1_{\mc{S}_{1_K}}(G_{K_v}, A[\lambda]) + H^1_{\mc{S}_\chi}(G_{K_v}, A[\lambda]),
\]
to prove the claim, it suffices to prove that $\phi$ is surjective onto $\ker(\mres)$.  But the restriction of $\mres$ to $H^1_{\mc{S}_{1_K}}(G_{K_v}, A[\lambda])$ is injective by Lemma \ref{size_and_ramified}(2), so
\begin{align*}
\dim_{\F_p}(\ker(\mres))
&= \dim_{\F_p}(H^1(G_{K_v}, A[\lambda])) - \dim_{\F_p}(\Ima(\mres)) \\
&\le \dim_{\F_p}(H^1(G_{K_v}, A[\lambda])) - \dim_{\F_p}(H^1_{\mc{S}_{1_K}}(G_{K_v}, A[\lambda])) \\
&= \dim_{\F_p}(H^1_{\mc{S}_{1_K}}(G_{K_v}, A[\lambda])) \\
&= \dim_{\F_p}(A(K_v)[\lambda]).
\end{align*}
Thus surjectivity will follow if we prove that $\phi$ is injective.

To see injectivity, let $\phi'$ be the composition of $\phi$ with the restriction map
\[
H^1(G_{K_v}, A[\lambda]) \ra H^1(G_{K_v(\zeta_{p^n})}, A[\lambda]),
\]
where $\zeta_{p^n}$ is a primitive $p^n$-th root of unity.  It suffices to prove that $\phi'$ is injective.  Easily
\[
\Ima(\phi') \subset \ker\left(\mres: H^1(G_{K_v(\zeta_{p^n})}, A[\lambda]) \ra H^1(G_{F_w(\zeta_{p^n})}, A[\lambda]) \right).
\]
Then by inflation-restriction, we can treat $\phi'$ as a map
\[
\phi': A(K_v)[\lambda] \ra H^1(\Gal(F_w(\zeta_{p^n}) / K_v(\zeta_{p^n})), A(F_w(\zeta_{p^n}))[\lambda]).
\]
By Lemma \ref{size_and_ramified}(2) and the fact that $F_w(\zeta_{p^n})/K_v(\zeta_{p^n})$ is totally ramified, $A(F_w(\zeta_{p^n}))[\lambda] = A(K_v(\zeta_{p^n}))[\lambda]$.  Thus
\begin{align*}
&H^1(\Gal(F_w(\zeta_{p^n}) / K_v(\zeta_{p^n})), A(F_w(\zeta_{p^n}))[\lambda]) \\ &= \Hom(\Gal(F_w(\zeta_{p^n}) / K_v(\zeta_{p^n})), A(K_v(\zeta_{p^n}))[\lambda]).
\end{align*}

For $\tau \in \Gal(F_w(\zeta_{p^n}) / K_v(\zeta_{p^n}))$ and $P \in A(K_v)[\lambda]$, we can compute $\phi'(P)(\tau)$ as follows.  Choose a preimage $Q \in A(\bar{K}_v)$ of $P$ under $\lambda$ and a preimage $\sigma \in G_{K_v(\zeta_{p^n})}$ of $\tau$ under the natural projection.  Then by the definitions of the Kummer map and of the twist by $\chi$,
\begin{align*}
\phi'(P)(\tau)
&= (\sigma . Q - Q) - \psi \left(\sigma . (\psi^{-1}(Q)) - \psi^{-1}(Q)\right) \\
&= \sigma . Q - \psi(\sigma . (\psi^{-1}(Q))) \\
&= \sigma . Q - \chi(\sigma)(Q).
\end{align*}
Next, by inflation-restriction, we can consider $\chi \in \Hom(\Gal(F_w(\zeta_{p^n}) / K_v(\zeta_{p^n})), H)$.  Then $\chi(\tau)$ is well-defined.  Also, $K_v(\zeta_{p^n}, \lambda^{-1}(P))/K_v(\zeta_{p^n})$ is Galois, and it is unramified because $\lambda^{-1}(P) \subset A[p]$, $v \nmid p$, and $A$ has good reduction at $v$, hence $K_v(\zeta_{p^n}, \lambda^{-1}(P)) \cap F_w(\zeta_{p^n}) = K_v(\zeta_{p^n})$.  Thus we can choose $\sigma$ lifting $\tau$ such that $\sigma . Q = Q$.  Then we can write
\[
\phi'(P)(\tau) = Q - \chi(\tau)(Q).
\]
By the assumption $F_w = \bar{K}_v^{\ker(\chi)}$, $\chi \in \Hom(\Gal(F_w(\zeta_{p^n}) / K_v(\zeta_{p^n})), H)$ is injective, hence also surjective.  Thus setting $\tau = \chi^{-1}(\eta)$, we have $\ker(1 - \chi(\tau)) = \ker(\lambda)$.  Then for all $P \neq 0$, $\phi'(P)(\tau) \neq 0$ because $\lambda(Q) = P \neq 0$, so $\phi'(P)$ is a nontrivial homomorphism.
\end{proof}

\subsection{The case $H \cong \mu_{p^n}$}\label{mu_n}
Fix $n \in \mb{N}$, and suppose $H \cong \mu_{p^n}$ as $G_K$-modules.  Then we have the Kummer isomorphism $H^1(G_K, H) \cong H^1(G_K, \mu_{p^n}) \cong K^*/(K^*)^{p^n}$.  Using this identification, for any $d \in K^*$ (implicitly taken modulo ${p^n}$-th powers), we can define the twist $A_d$ of $A$ by $d$ over $K$, as well as the local invariant $\delta_v(A, d)$.

% Now if $c \equiv d \pmod{(K_v^*)^{p^n}}$, then $\delta_v(A, c) = \delta_v(A, d)$ by definition.  Combining this observation with Theorem \ref{parity_twist_formula} and Lemmas \ref{unramified} and \ref{totally_ramified} gives our main result for this section.

Now combining Theorem \ref{parity_twist_formula} and Lemmas \ref{unramified} and \ref{totally_ramified}, we get our main result for this section.

\begin{thm}\label{main_bg}
Assume $\lambda$ satisfies conditions (\ref{poonen_rains_assumption}) and (\ref{yu_assumption}), assume $H \cong \mu_{p^n}$ as $G_K$-modules, and assume $\ker(1 - \eta) = \ker(\lambda)$ for some generator $\eta$ of $H$.  Let $d \in K^*$.  Then
\begin{align*}
d_\lambda(A_d/K) - d_\lambda(A/K) \equiv& \sum_{v \in \Sigma} \delta_v(A, d) + \sum_{\substack{v \notin \Sigma:\\ p \nmid \ord_v(d)}} \dim_{\mb{F}_p}(A(K_v)[\lambda]) \\ &+ \sum_{\substack{v \notin \Sigma :\\ p \mid \ord_v(d), \\ p^n \nmid \ord_v(d)}} \delta_v(A, d) \pmod{2}.
\end{align*}
\end{thm}

\section{Quartic Twists of Elliptic Curves with $j$-Invariant 1728}\label{quartic}
Fix a number field $K$.  Let $E/K$ be the elliptic curve with affine equation $y^2 = x^3 + x$, and let $E'/K$ be the elliptic curve $y^2 = x^3 - 4x$.  We wish to classify the conjectural rank parities of quartic twists of $E/K$, which have the form $E_d/K: y^2 = x^3 + dx$ for $d \in K^*/(K^*)^4$.  Note that these are all of the elliptic curves over $K$ with $j$-invariant 1728.

We have dual degree two isogenies (\cite[Example III.4.5]{silverman})
\begin{align*}
\varphi: E_d &\rightarrow E'_d
&\hat{\varphi}: E'_d &\rightarrow E_d \\
(x, y) &\mapsto \left( \frac{y^2}{x^2}, \frac{y(d-x^2)}{x^2} \right)
&(x', y') &\mapsto \left( \frac{y'^2}{4x'^2}, \frac{y'(-4d-x'^2)}{8x'^2} \right).
\end{align*}
From these, we can construct a self-dual degree two isogeny
\begin{align*}
\lambda: E_d \times E'_d &\rightarrow E_d \times E'_d \\
(P, Q) &\mapsto (\hat{\varphi}(Q), \varphi(P)).
\end{align*}
Set $A := E \times E'$, from which $A_d = E_d \times E'_d$.  Note that the kernel of $\lambda: A_d \ra A_d$ is
\[
A_d[\lambda] = E_d[\varphi] \times E'_d[\hat{\varphi}] = \{O, (0, 0)\} \times \{O', (0, 0)\},
\]
and it is defined over $K$.  By the following result, for our problem, it suffices to study how the $\lambda$-Selmer rank of $A$ varies with $d$.

\begin{lemma}
\label{tau_to_2}
Assuming the Shafarevich-Tate conjecture, for any $d \in K^*$,
\[
d_\lambda(A_d/K)) \equiv \rk(E_d/K) \pmod{2}.
\]
\end{lemma}
\begin{proof}
We have the well-known exact sequence
\begin{align*}
0 \ra &\frac{E'_d(K)[\hat{\varphi}]}{\varphi(E_d(K)[2])} \ra \msel^{(\varphi)}(E'_d/K) \ra \msel^{(2)}(E_d/K) \\ &\ra \msel^{(\hat{\varphi})}(E_d/K) \ra \frac{\Sha(E'_d/K)[\hat{\varphi}]}{\varphi(\Sha(E_d/K)[2])} \ra 0.
\end{align*}
Under our assumption, properties of the Cassels pairing imply that the last term has even dimension \cite[Corollary to Theorem 1.2]{cassels_pairing_results}.  Then using the fact that $\msel^{(\lambda)}(A_d/K) \cong \msel^{(\varphi)}(E_d/K) \times \msel^{(\hat{\varphi})}(E'_d/K)$, we have
\begin{align*}
d_\lambda(A_d/K)
&\equiv d_2(E_d/K) + \dim_{\F_2}\left(\frac{E'_d(K)[\hat{\varphi}]}{\varphi(E_d(K)[2])} \right) \\
&\equiv d_2(E_d/K) + \dim_{\F_2}(E_d(K)[2]) \\ &\quad + \dim_{\F_2}(E'_d(K)[\hat{\varphi}]) + \dim_{\F_2}(E_d(K)[\varphi]) \\
&\equiv \rk(E_d/K) + \dim_{\F_2}(E'_d(K)[\hat{\varphi}]) + \dim_{\F_2}(E_d(K)[\varphi]) \\
&\equiv \rk(E_d/K) \pmod{2}
\end{align*}
using the explicit description of $E_d[\varphi]$ and $E'_d[\hat{\varphi}]$ given above.
\end{proof}

\begin{rem}
\label{rem_mu_4}
Suppose $\mu_4 \subset K$.  Then $E_d \cong E_d'$ over $K$, and the isomorphism identifies $\varphi$ with $\hat{\varphi}$.  Hence
\[
d_\lambda(A_d/K) = 2d_\varphi(E_d/K) \equiv 0 \pmod{2},
\]
so that conjecturally $\rk(E_d/K) \equiv 0 \pmod{2}$ for all values of $d$ by Lemma \ref{tau_to_2}.  Thus we are primarily interested in the case $\mu_4 \not\subset K$.
\end{rem}

To solve our problem for general $K$, we wish to apply Theorem \ref{main_bg} to $A/K$ and $\lambda$, with $p = 2$.  Note that the finite places of bad reduction for $A$ are at most those dividing $2$.  For our subgroup $H$ of $\mbox{Aut}(A)$, we choose the diagonal embedding of $\mu_4$ into $\mbox{Aut}(E) \times \mbox{Aut}(E')$ (recall $\mbox{Aut}(E) \cong \mbox{Aut}(E') \cong \mu_4$ by \cite[Corollary III.10.2]{silverman}).  Thus $\zeta \in \mu_4$ acts on $A$ as
\[
[\zeta]((x, y), (x', y')) = ((\zeta^2x, \zeta^3y), (\zeta^2x', \zeta^3y')).
\]
By our description of $A[\lambda]$ given above, $H$ acts trivially on $A[\lambda]$, and $\ker(1 - [\zeta]) = \ker(\lambda)$ for any generator $\zeta$ of $H$.  Also, $H$ commutes with $\lambda$.

Now $A$ is self-dual since $E$ and $E'$ are.  Since $p = 2$, we must additionally verify conditions (\ref{poonen_rains_assumption}) and (\ref{yu_assumption}).  We can identify $\Pic_0(A)$ with the group of Weil divisors of $A$ modulo principal divisors.  The natural isomorphism $A \xra{\sim} \Pic_0(A)$ is then given by
\[
(P, Q) \mapsto [P \times E'] - [O \times E'] + [E \times Q] - [E \times O'],
\]
where $O$ and $O'$ denote the identity elements of $E$ and $E'$, respectively.

Let $\Gamma(-\varphi) = \{(R, -\varphi(R)) \mid R \in E\}$ be the graph of $-\varphi$ in $E \times E'$.  Then $[\Gamma(-\varphi)]$ is a divisor on $A$.  Using the seesaw principle (\cite[p.\ 241]{lang}) and the above isomorphism $A \cong \Pic_0(A)$, one can see that
\begin{align*}
\phi_{[\Gamma(-\varphi)]}(P, Q) &= (2P + \hat{\varphi}(Q), Q + \varphi(P)),
\end{align*}
since the corresponding divisors are equivalent when restricted to $U \times E'$ or $E \times V$ for any $U \in E$ or $V \in E'$.  Also, $\phi_{[O \times E']}(P, Q) = \phi_{[(0, 0) \times E']}(P, Q) = (P, O')$ and $\phi_{[E \times O']}(P, Q) = (O, Q)$.  Thus letting
\[
D := [\Gamma(-\varphi)] - [O \times E'] - [(0, 0) \times E'] - [E \times O'],
\]
we have $\phi_D = \lambda$.  Obviously $D$ is symmetric, so this proves condition (\ref{poonen_rains_assumption}).

We now verify condition (\ref{yu_assumption}).  Let $\bar{K}[A]$ and $\bar{K}(A)$ denote the ring of regular functions on $A$ and the field of rational functions on $A$, respectively.  For $g \in \bar{K}(A)$, let $\mdiv(g) \in \Pic_0(A)$ denote the Weil divisor corresponding to $g$.  For $X \in A$, let $T_X$ denote the translation-by-$X$ map.  Then in terms of Weil divisors, we can write the Heisenberg group $\mathcal{H}_A(D)$ used in the construction of the canonical global metabolic structure (see \cite[\S 4]{poonen_rains}) as
\[
\mathcal{H}_A(D) = \{(X, g) \mid \mbox{$X \in A[\lambda]$, $g \in \bar{K}(A)$, and $\mdiv(g) = T_X^*(D) - D$}\},
\]
with group operation given by $(X, g)(X', g') = (X + X', T_{X'}^*(g)g')$.

\begin{lemma}
\label{heisenberg_group_fixed_by_i}
Let $\zeta_4$ be a primitive 4th root of unity, which acts on $A$ as above.  Let $(X, g) \in \mathcal{H}(D)$, so that $\mdiv(g) = T_X^*(D) - D$.  Then $g \circ [\zeta_4] = g$.
\end{lemma}
\begin{proof}
We adapt the proof of \cite[Lemma 5.9]{yu}.  When $X = (O, O')$ or $((0, 0), O')$, we have $T_X^*(D) - D = 0$, so $g \in \bar{K}[A] = \bar{K}[E] \otimes_{\bar{K}} \bar{K}[E'] = \bar{K}$.  Then $g$ is constant and the conclusion is trivial.

When $X$ is $(O, (0, 0))$ or $((0, 0), (0, 0))$, first note that
\[
\mdiv(g \circ [\zeta_4]) = [\zeta_4]^*(\mdiv(g)) = [\zeta_4]^*(T_X^*(D) - D) = T_X^*(D) - D = \mdiv(g),
\]
where $\zeta_4$ acts trivially on $X$ because it is in $A[\lambda]$, and where $[\zeta_4]^*D = D$ because $-\varphi \circ (\zeta_4|_E) = (\zeta_4|_{E'}) \circ -\varphi$.  Hence $g \circ [\zeta_4] = cg$ for some $c \in \bar{K}[A]^* = \bar{K}^*$.  Necessarily $c \in \mu_4$.

To show $c = 1$, we claim that it suffices to find a nonsingular curve $C$ on $A$ such that:
\begin{itemize}
  \item $[\zeta_4] C = C$.
  \item The divisor $[C]$ has order 0 in $\mdiv(g)$.
  \item Letting $\mdiv_C(g|_C)$ denote the divisor of $g|_C$ on $C$, there is some $Y \in C \cap A[\lambda]$ such that the order of $Y$ in $\mdiv_C(g|_C)$ is an integer multiple of four.
\end{itemize}

We prove this claim as follows.  Write $\widehat{\mc{O}_C}$ for the completion of the local ring $\mc{O}_C$ of $A$ at $C$ (over $\bar{K}$).  Then $\widehat{\mc{O}_C} \cong \bar{K}(C)[[t]]$, where $\bar{K}(C)$ is the function field of $C/\bar{K}$ and $t$ is a uniformizer for $C$. %(see pg.\ 8 of handwritten notes).
We have an automorphism $[\zeta_4]^* \in \Aut(\bar{K}(C)[[t]])$ induced by $[\zeta_4]$, due to the assumption that $[\zeta_4] C = C$.  Since this automorphism has order four, it sends $t$ to $\zeta t + (\mbox{higher order terms})$ for some $\zeta \in \mu_4$.  Because $\ord_C(g) = 0$ by assumption, we have
\[
g = f_0 + (\mbox{higher order terms})
\]
as a power series in $\bar{K}(C)[[t]]$.  Then
\[
[\zeta_4]^* g = [\zeta_4]^* f_0 + (\mbox{higher order terms}),
\]
where the action of $[\zeta_4]^*$ on $f_0$ is induced by the action of $[\zeta_4]$ on $C$.  Note that $f_0 = g|_C$ since the higher order terms vanish on $C$.

We now repeat this process for $f_0$: let $\widehat{\mc{O}^C_Y} \cong \bar{K}[[u]]$ be the completion of the local ring of $C$ at $Y$, for some uniformizer $u$ at $Y$, and write
\[
f_0 = d u^n + (\mbox{higher order terms})
\]
for some $d \in \bar{K}$, with $n = \ord_Y(f_0) = \ord_Y(g|_C)$.  By assumption, $4 | n$.  The automorphism $[\zeta_4]^*$ of $\bar{K}[[u]]$ sends $u$ to $\zeta' u + (\mbox{higher order terms})$ for some $\zeta' \in \mu_4$, so
\[
[\zeta_4]^* f_0 = d u^n + (\mbox{higher order terms}).
\]
But this shows that $[\zeta_4]^* f_0$ and $f_0$ have the same leading term in $\bar{K}[[u]]$, hence the same is true of $g$ as an element of $\bar{K}(C)[[t]]$.  Thus $c = 1$ and $g \circ [\zeta_4] = g$, as claimed.

It is now easy to verify that the hypotheses of the claim hold when we take $C = [O \times E']$.  Indeed, for either $X$, we have
\[
T_X^*(D) - D = [\Gamma(-\varphi + (0, 0))] - [E \times (0, 0)] - [\Gamma(-\varphi)] + [E \times O'].
\]
Then $C$ appears with order zero.  For any $g$ with $\mdiv(g) = T_X^*(D) - D$, we can compute $\mdiv_C(g|_C)$ by intersecting each irreducible divisor with $C$, yielding
\[
(T_X^*(D) - D)|_C = [(O, (0, 0))] - [(O, (0, 0))] - [(O, O')] + [(O, O')] = 0.
\]
Hence $Y = (O, O') \in C$ appears with coefficient a multiple of four.
\end{proof}

\begin{lemma}
\label{quartic_twist_quadratic_forms}
Let $v$ be a place of $K$, and let $d \in K_v^*/(K_v^*)^4$.  Then the natural isomorphism
\[
H^1(G_{K_v}, A_d[\lambda]) \xra{\sim} H^1(G_{K_v}, A[\lambda])
\]
identifies the canonical quadratic forms constructed in \cite[\S 4]{poonen_rains}.  Hence $\lambda$ satisfies condition (\ref{yu_assumption}).
\end{lemma}
\begin{proof}
This follows by the same argument as for \cite[Lemma 5.2(ii)]{kbr}.  Indeed, we have a commutative diagram
\[
\begindc{\commdiag}[80]
\obj(0, 5)[one]{0}
\obj(5, 5)[a]{$\bar{K}_v^\times$}
\obj(11, 5)[b]{$\mathcal{H}_A(D)$}
\obj(17, 5)[c]{$A[\lambda]$}
\obj(22, 5)[one2]{0}
\mor{one}{a}{}
\mor{a}{b}{}
\mor{b}{c}{}
\mor{c}{one2}{}

\obj(0,0)[oneV]{0}
\obj(5, 0)[aV]{$\bar{K}_v^\times$}
\obj(11, 0)[bV]{$\mathcal{H}_{A_d}(D_d)$}
\obj(17, 0)[cV]{$A_d[\lambda]$}
\obj(22, 0)[one2V]{0}
\mor{oneV}{aV}{}
\mor{aV}{bV}{}
\mor{bV}{cV}{}
\mor{cV}{one2V}{}

\mor{aV}{a}{}[\atright, \equalline]
\mor{bV}{b}{$\sim$}
\mor{cV}{c}{$\sim$}
\enddc
\]
By the fact that $H$ fixes $A[\lambda]$ and the previous lemma, these isomorphisms are all fixed by $G_{K_v}$.  Hence they give an isomorphism of the corresponding long exact sequences from Galois cohomology.  But the canonical quadratic forms are defined to be the connecting morphisms $H^1(G_{K_v}, A[\lambda]) \ra H^2(G_{K_v}, \bar{K}_v^\times)$ and $H^1(G_{K_v}, A_d[\lambda]) \ra H^2(G_{K_v}, \bar{K}_v^\times)$ in these long exact sequences, so they are identified by the natural isomorphism.
\end{proof}

Now Theorem \ref{main_bg} says that we will be done once we compute the local invariants $\delta_v(A, d)$ when $v \notin \Sigma$ and $\ord_v(d) \equiv 2 \pmod{4}$.  This is accomplished in the following lemma.

\begin{lemma}
\label{half_ramified}
Assume $v \nmid 2\infty$, $A$ has good reduction at $v$, and $\ord_v(d) \equiv 2 \pmod{4}$.  Then $\delta_v(A, d) \equiv 0 \pmod{2}$ iff $\legendre{-1}{v} = 1$.
\end{lemma}
\begin{proof}
Let $c \in \Or_v^*$ be such that $c\pi_v^2 \equiv d \pmod{(K_v^*)^4}$.  Using \cite[Corollary 2.5]{kbr} and noting that we can define $\delta_v(A_c, \pi_v^2)$ in a natural way, we have
\[
\delta_v(A, c\pi_v^2) \equiv \delta_v(A, c) + \delta_v(A_c, \pi_v^2) \pmod{2}.
\]
By Lemma \ref{unramified}, $\delta_v(A, c) = 0$, so $\delta_v(A, d) \equiv \delta_v(A_c, \pi_v^2) \pmod{2}$.

For any $a \in K_v^*$, let $H^1_{\mathcal{S}_a}(G_{K_v}, A_c[2]) \subset H^1(G_{K_v}, A_c[2])$ denote the $v$-part of the Selmer structure corresponding to the multiplication-by-2 isogeny on the \textit{quadratic} twist of $A_c$ by $a$.  Then we claim that
\begin{equation}\label{half_ramified_sub_claim}
\delta_v(A_c, \pi_v^2) \equiv \dim_{\F_2}(H^1_{\mathcal{S}_{1_K}}(G_{K_v}, A_c[2])) \pmod{2}.
\end{equation}

To see this, consider the map
\[
\iota': H^1_{\mathcal{S}_{1_K}}(G_{K_v}, A_c[2]) \rightarrow \frac{H^1_{\mathcal{S}_{1_K}}(G_{K_v}, A_c[\lambda])}{H^1_{\mathcal{S}_{1_K}}(G_{K_v}, A_c[\lambda]) \cap H^1_{\mathcal{S}_{\pi_v^2}}(G_{K_v}, A_c[\lambda])}
\]
induced by the identity map
\[
\iota: A_c(K_v)/2A_c(K_v) \rightarrow A_c(K_v)/\lambda(A_c(K_v)).
\]
This map is surjective and has kernel
\[
\ker(\iota') = \left( H^1_{\mathcal{S}_{1_K}}(G_{K_v}, A_c[2]) \cap H^1_{\mathcal{S}_{\pi_v}}(G_{K_v}, A_c[2]) \right) + i_c(\ker(\iota)),
\]
where $i_c$ is the Kummer map.  But
\[
H^1_{\mathcal{S}_{1_K}}(G_{K_v}, A_c[2]) \cap H^1_{\mathcal{S}_{\pi_v}}(G_{K_v}, A_c[2]) = 0
\]
by Lemma \ref{totally_ramified} applied to the quadratic twist of $A_c$ by $\pi_v$.  Hence
\[
\delta_v(A_c, \pi_v^2) = \dim_{\F_2}(H^1_{\mathcal{S}_{1_K}}(G_{K_v}, A_c[2])) - \dim_{\F_2}(\ker(\iota)).
\]
Finally,
\[
\ker(\iota) = \frac{\lambda(A_c(K_v))}{2A_c(K_v)} \cong \frac{A_c(K_v)}{\lambda(A_c(K_v))}
\]
has dimension $\dim_{\F_2}(A_c(K_v)[\lambda]) = 2$ by Lemma \ref{size_and_ramified}(1).  This proves (\ref{half_ramified_sub_claim}).

To finish the proof, note that
\[
\dim_{\F_2}(H^1_{\mathcal{S}_{1_K}}(G_{K_v}, A_c[2])) = \dim_{\F_2}(A_c[2]) = \dim_{\F_2}(E_c[2]) + \dim_{\F_2}(E'_c[2])
\]
by Lemma \ref{size_and_ramified}(1) applied to the multiplication-by-2 isogeny.  Then easily
\begin{align*}
&\dim_{\F_2}(E_c[2]) + \dim_{\F_2}(E'_c[2]) \equiv 0 \pmod{2} \\
\iff &\legendre{-c}{v}\legendre{4c}{v} = 1 \\
\iff &\legendre{-1}{v} = 1.
\end{align*}
\end{proof}

Finally, we get our main result for this section.

% \begin{thm}
% \label{main_thm}
% Let $c, d \in K^*$ be such that for all places $v | 2\infty$, $c \equiv d \pmod{(K_v^*)^4}$.  Then
% \begin{align*}
% &d_\lambda(A_c/K) \equiv d_\lambda(A_d/K) \pmod{2} \\
% \iff &\legendre{-1}{S(c)} = \legendre{-1}{S(d)},
% \end{align*}
% where $S(a) \subset \Or_K$ is defined to be the product of all places $v \nmid 2\infty$ of $K$ such that $\ord_v(a) \equiv 2 \pmod{4}$.  Assuming the Shafarevich-Tate conjecture, the same holds for $\rk(E_c/K) \equiv \rk(E_d/K) \pmod{2}$.
% \end{thm}

\begin{thm}
\label{main_thm}
Let $d \in K^*$.  Then
\begin{align*}
&d_\lambda(A_d/K) - d_\lambda(A/K) \equiv \sum_{v | 2\infty} \delta_v(A, d) + \frac{1 - \legendre{-1}{S(d)}}{2} \pmod{2},
\end{align*}
where $S(d) \subset \Or_K$ is defined to be the product of all places $v \nmid 2\infty$ of $K$ such that $\ord_v(d) \equiv 2 \pmod{4}$.  Assuming the Shafarevich-Tate conjecture, the same holds for $\rk(E_d/K) - \rk(E/K) \pmod{2}$.
\end{thm}

\begin{rem}
Note that each of the local invariants $\delta_v(A, d)$ depends only on the value of $d \pmod{(K_v^*)^4}$: indeed, if $c \equiv d \pmod{(K_v^*)^4}$, then $A_c \cong A_d$ over $K_v$.  Hence the above theorem and the statements below allow us to classify $d_\lambda(A_d/K) \pmod{2}$ for all $d \in K^*$ after performing a finite amount of computation, namely, that required to compute $\delta_v(A, d)$ for all $v | 2\infty$, $d \in K_v^* / (K_v^*)^4$.
\end{rem}

Obviously if $v$ is real, then $c \equiv d \pmod{(K_v^*)^4}$ iff $c$ and $d$ have the same sign at $v$.  For places $v | 2$, the following fact is useful.

\begin{prop}\label{mod_4_prop}
Let $v | 2$ be a place of $K$, and let $c, d \in K_v^*$.  Define the integer $m_v$ by $m_v := 3e_{v/2} + 1$, where $e_{v/2}$ is the ramification index of $v$ over 2.  Suppose that
\begin{itemize}
  \item $\ord_v(c) \equiv \ord_v(d) \pmod{4}$
  \item $c/(\pi_v^{\ord_v(c)})$ and $d/(\pi_v^{\ord_v(d)})$ have the same residues in $(\Or_v/v^{m_v})^*/((\Or_v/v^{m_v})^*)^4$.
\end{itemize}
Then $c \equiv d \pmod{(K_v^*)^4}$.
\end{prop}
\begin{proof}
Without loss of generality, $c$ and $d$ are units in $\mc{O}_v$ which have the same residues in $(\Or_v/v^{m_v})^*/((\Or_v/v^{m_v})^*)^4$.  Then $c \equiv d \pmod{(K_v^*)^4}$ follows from a straightforward modification of \cite[Lemma 4.1 and Corollary 4.3]{quadratic_twists}, as follows.

First, we show that if $0 \neq \alpha \in \mc{O}_v$ is such that $\alpha \equiv 1 \pmod{v^{m_v}}$, then $\alpha \in (K_v^*)^4$.  Write $\alpha = 1 + \beta \pi_v^n$ with $\beta \in \mc{O}_v$ and $n \ge m_v$.  Consider the binomial expansion
\[
\alpha^{1/4} = (1+ \beta\pi_v^n)^{1/4} = \sum_{i=0}^\infty \binom{1/4}{i} \beta^i \pi_v^{in}.
\]
This series converges in $K_v$ so long as the valuations of the terms appearing in the right-hand series tend to infinity.  By Legendre's formula, we have
\begin{align*}
\ord_v\binom{1/4}{i}
&= \ord_v\left( \frac{1(1-4)(1-8) \cdots (1 - 4(i-1))}{4^i i!}\right) \\
&\ge -\left(2i + \frac{i-1}{2-1}\right) \ord_v(2) \\
&= -(3i - 1)e_{v/2}.
\end{align*}
Hence we want
\begin{align*}
\lim_{i\ra \infty} \left( in - (3i-1)e_{v/2} \right) \ra \infty,
\end{align*}
which happens precisely because $n \ge m_v = 3e_{v/2} + 1$.

Next, the image of $c/d$ in $(\Or_v/v^{m_v})^*$ is a fourth power, hence has some preimage $b\in\Or_v$ which is a fourth power in $\Or_v$.  Then $(c/d)/b \equiv 1 \pmod{v^{m_v}}$, so it is a fourth power in $\Or_v$ by the previous paragraph, hence so is $c/d$.
\end{proof}

\section{Sextic Twists of Elliptic Curves with $j$-Invariant 0}\label{sextic}
The sextic twist case begins analogously to the quartic twist case.  Fix a number field $K$.  Let $E/K$ be the elliptic curve with affine equation $y^2 = x^3 + 1$, and let $E'/K$ be the elliptic curve $y^2 = x^3 - 27$.  We wish to classify the conjectural rank parities of sextic twists of $E/K$, which have the form $E_d/K: y^2 = x^3 + d$ for $d \in K^*/(K^*)^6$.  Note that these are all of the elliptic curves over $K$ with $j$-invariant 0.

We have dual degree three isogenies (found using SAGE \cite{sage, pari})
\begin{align*}
\varphi: E_d &\rightarrow E'_d
&\hat{\varphi}: E'_d &\rightarrow E_d \\
(x, y) &\mapsto \left( \frac{x^3+4d}{x^2}, \frac{y(x^3-8d)}{x^3} \right)
&(x', y') &\mapsto \left( \frac{x'^3-108d}{9x'^2}, \frac{y'(x'^3+216d)}{27x'^3} \right).
%&(x', y') &\mapsto \left( \frac{(1/9)x'^3-12d}{x'^2}, \frac{y'((1/27)x'^3+8d)}{x'^3} \right)
\end{align*}
From these, we can construct a self-dual degree three isogeny
\begin{align*}
\lambda: E_d \times E'_d &\rightarrow E_d \times E'_d \\
(P, Q) &\mapsto (\hat{\varphi}(Q), \varphi(P)).
\end{align*}
Set $A := E \times E'$, from which $A_d = E_d \times E'_d$.  Note that the kernel of $\lambda: A_d \ra A_d$ is
\[
A_d[\lambda] = E_d[\varphi] \times E'_d[\hat{\varphi}] = \{O, (0, \pm \sqrt{d})\} \times \{O', (0, \pm \sqrt{-27d})\}.
\]
Thus $\dim(A(K)[\lambda]) \equiv 0 \pmod{2}$ iff $\sqrt{-27d^2} \in K$, which is iff $\sqrt{-3} \in K$.

By the following result, for our problem, it suffices to study how the $\lambda$-Selmer rank of $A$ varies with $d$.

\begin{lemma}
\label{tau_to_3}
Assuming the Shafarevich-Tate conjecture,
\begin{align*}
&d_\lambda(A_d/K) \equiv \rk(E_d/K) \pmod{2} \\
&\iff \mu_3 \subset K.
\end{align*}
\end{lemma}
\begin{proof}
By the same argument as in the quartic twist case, $d_\lambda(A_d/K) \equiv 0 \pmod{2}$ iff $\dim(A(K)[\lambda]) \equiv 0 \pmod{2}$, which is iff $\sqrt{-3} \in K$.  Since $\zeta_3 = \frac{1}{2}(-1\pm \sqrt{-3})$, this is iff $\mu_3 \subset K$.
\end{proof}

\begin{rem}
\label{rem_mu_3}
Suppose $\mu_3 \subset K$.  Then $\sqrt{-3} \in K$, so $E_d \cong E'_d$ over $K$, and the isomorphism identifies $\varphi$ with $\hat{\varphi}$.  Hence
\[
d_\lambda(A_d/K) = 2d_\varphi(E_d/K) \equiv 0 \pmod{2},
\]
so that conjecturally $\rk(E_d/K) \equiv 0 \pmod{2}$ for all values of $d$ by Lemma \ref{tau_to_3}.  Thus we are primarily interested in the case $\mu_3 \not\subset K$.
\end{rem}

To solve our problem for general $K$, we wish to apply the results of Section \ref{background} to $A/K$ and $\lambda$, with $p = 3$.  However, this will not give us a full solution --- indeed, the diagonal embedding of $\mu_6$ into $\mbox{Aut}(E)$ does not fix $A[\lambda]$; only the subgroup $\mu_3$ does, and there is no nontrivial isogeny whose kernel is fixed by $\mu_6$.  Hence our background results will only apply to the cubic twists $A_{d^2}$ of $A/K$, for $d \in K^*/(K^*)^3$.  Nonetheless, we will see in Theorem \ref{main_thm2} that we can use cubic twists to understand sextic twists as well.

For our subgroup $H$ of $\mbox{Aut}(A)$, we choose the subgroup $\mu_3$ of the diagonal embedding of $\mu_6$ into $\mbox{Aut}(E) \times \mbox{Aut}(E')$.  Thus $\zeta \in \mu_3$ acts on $A$ as
\[
[\zeta]((x, y), (x', y')) = ((\zeta^2x, y), (\zeta^2x', y')).
\]
By our description of $A[\lambda]$ given above, $H$ acts trivially on $A[\lambda]$, and $\ker(1 - [\zeta]) = \ker(\lambda)$ for any generator $\zeta$ of $H$.  Also, $H$ commutes with $\lambda$.

Now $A$ is self-dual since $E$ and $E'$ are.  Also, $p = 3$ is odd, so we do not need any extra geometric conditions on $\lambda$.

Hence Theorem \ref{main_bg} gives the following result for cubic twists of $A$.  Note that the finite places of bad reduction for $A/K$ are at most those dividing 6.

\begin{prop}
\label{cubic_twists}
Let $d \in K^*$.  Then
\begin{align*}
&d_\lambda(A_{d^2}/K) - d_\lambda(A/K) \equiv \sum_{v | 6\infty} \delta_v(A, d) + \frac{1 - \legendre{-3}{T(d)}}{2} \pmod{2},
\end{align*}
where $T(d) \subset \Or_K$ is defined to be the product of all places $v \nmid 6\infty$ of $K$ such that $3 \nmid \ord_v(d)$.  Assuming the Shafarevich-Tate conjecture, the same holds for $\rk(E_{d^2}/K) - \rk(E/K) \pmod{2}$.
\end{prop}

To extend this result to all sextic twists, recall that if $B$ is an abelian variety and $B_d$ is the quadratic twist of $B$ by some nontrivial $d \in K^*/(K^*)^2$, then $\rk(B/K) + \rk(B_d/K) = \rk(B/K(\sqrt{d}))$.  Hence we can relate the rank of the sextic twist by $d$ over $K$ to that of the sextic twist by $d$ over $K(\sqrt{d})$, i.e., the cubic twist by $\sqrt{d}$ over $K(\sqrt{d})$, whose rank parity we in principle understand by the above proposition.  Using this idea, we get the following result for sextic twists.

\begin{thm}
\label{main_thm2}
Assume the Shafarevich-Tate conjecture.  Let $d \in K^*$.  Then there exist local invariants $\epsilon_v(d)$, defined for each place $v | 6\infty$ and depending only on the value of $d \pmod{(K_v^*)^6}$, such that
\begin{align*}
&\rk(E_d/K) - \rk(E/K) \equiv \sum_{v | 6\infty} \epsilon_v(d) + \frac{1 - \legendre{-3}{U(d)}}{2} \pmod{2},
\end{align*}
where $U(d) \subset \Or_K$ is defined to be the product of all places $v \nmid 6\infty$ of $K$ such that $\ord_v(d) \equiv 2, 4 \pmod{6}$.
\end{thm}
\begin{proof}
From
\begin{align*}
\rk(E_d/K) + \rk(E_{d^4}/K) &= \rk(E_d/K(\sqrt{d})) \\
\rk(E/K) + \rk(E_{d^3}/K) &= \rk(E/K(\sqrt{d})),
\end{align*}
we see that
\begin{align}
\label{eq_all_twists}
\rk(E_d/K) - \rk(E/K)
=& \Big(\rk(E_d/K(\sqrt{d}))-\rk(E/K(\sqrt{d}))\Big) \nonumber \\ &- \Big(\rk(E_{d^4}/K) - \rk(E/K)\Big) \nonumber \\ &+ \Big(\rk(E_{d^3}/K) - \rk(E/K)\Big).
\end{align}
By Proposition \ref{cubic_twists} and Lemma \ref{tau_to_3},
\begin{equation}
\label{eq_4th_powers}
\begin{aligned}
&\rk(E_{d^4}/K) - \rk(E/K) \equiv \sum_{v | 6\infty}  \delta_v(A, d^2) + \frac{1 - \legendre{-3}{T(d)}}{2} \pmod{2},
\end{aligned}
\end{equation}
where $T(d)$ is as in Proposition \ref{cubic_twists}.  By the quadratic twist case \cite[Proposition 7.2]{kbr}, there exist local invariants $\omega_v(d)$, depending only on the value of $d \pmod{(K_v^*)^2}$, such that
\begin{equation}
\label{eq_3rd_powers}
\rk(E_{d^3}/K) - \rk(E/K) \equiv \sum_{v | 6\infty} \omega_v(d) \pmod{2}.
\end{equation}

It remains to consider $\rk(E_d/K(\sqrt{d}))-\rk(E/K(\sqrt{d}))$.  Let $M_K$ and $M_{K(\sqrt{d})}$ denote the sets of all places of $K$ and $K(\sqrt{d})$, respectively.  By Theorem \ref{parity_twist_formula} and Lemma \ref{tau_to_3}, we have
\[
\rk(E_d/K(\sqrt{d}))-\rk(E/K(\sqrt{d})) \equiv \sum_{w \in M_{K(\sqrt{d})}} \delta_w(A/K(\sqrt{d}), d) \pmod{2}.
\]
To evaluate this sum, first let $v | 6\infty$, and suppose $c \equiv d \pmod{(K_v^*)^6}$.  Then the localizations of $K(\sqrt{c})$ and $K(\sqrt{d})$ at places lying above $v$ are the same.  Also, since $c \equiv d \pmod{(K_v^*)^3}$, we have $c \equiv d \pmod{(K(\sqrt{c})_w^*)^3}$ for any place $w$ of $K(\sqrt{c})$ lying above $v$.  Hence from the definition of the local invariants, we see that
\[
\delta_w(A/K(\sqrt{c}), c) = \delta_{w'}(A/K(\sqrt{d}), d)
\]
where $w$ is any place of $K(\sqrt{c})$ lying above $v$ and $w'$ is the corresponding place of $K(\sqrt{d})$.  Thus
\[
\eta_v(d) := \sum_{\substack{w \in M_K(\sqrt{d}) \\ w | v}} \delta_w(A/K(\sqrt{d}), d)
\]
is a local invariant depending only on the value of $d \pmod{(K_v^*)^6}$.

% Thus
% \begin{align*}
% &\rk(E_c/K(\sqrt{c}))-\rk(E/K(\sqrt{c})) \equiv \rk(E_d/K(\sqrt{d}))-\rk(E/K(\sqrt{d})) \pmod{2} \\
% \iff &\sum_{\substack{w \in M_{K(\sqrt{c})} \\ w \nmid 6\infty}} \delta_w(A/K(\sqrt{c}), c) \equiv \sum_{\substack{w \in M_{K(\sqrt{d})} \\ w \nmid 6\infty}} \delta_w(A/K(\sqrt{d}), d) \pmod{2}.
% \end{align*}

Next, by Lemmas \ref{unramified} and \ref{totally_ramified}, along with our explicit description of $A[\lambda]$,
\begin{align*}
&\sum_{\substack{w \in M_{K(\sqrt{d})} \\ w \nmid 6\infty}} \delta_w(A/K(\sqrt{d}), d) \equiv 0 \pmod{2} \\
\iff &\prod_{\substack{w \in M_{K(\sqrt{d})} \\ w \nmid 6\infty \\ 3 \nmid \ord_w(\sqrt{d})}} \legendre{-3}{w} = 1.
\end{align*}
We can rewrite this product as
\[
\prod_{\substack{v \in M_K \\ v \nmid 6\infty \\ 3 \nmid \ord_v(d)}} \left(\prod_{\substack{w \in M_{K(\sqrt{d})} \\ w | v}} \legendre{-3}{w}\right).
\]
To evaluate the interior products, break into cases depending on the factorization of $v$ in $K(\sqrt{d})$:
\begin{itemize}
 \item Case $v$ is split: then the two equal terms in the product cancel out.
 \item Case $v$ is unramified non-split: if $\legendre{-3}{v} = 1$, then $\legendre{-3}{w} = 1$.  Else $-3$ and $d$ are both quadratic non-residues modulo $v$, so adjoining a square root of $d$ to the residue field adjoins a square root of $-3$ as well.  Hence $\legendre{-3}{w} = 1$ in either case.
 \item Case $v$ is ramified: then since the extension of residue fields is trivial, $\legendre{-3}{w} = \legendre{-3}{v}$.
\end{itemize}
Then
\[
\prod_{\substack{v \in M_K \\ v \nmid 6\infty \\ 3 \nmid \ord_v(d)}} \left(\prod_{\substack{w \in M_{K(\sqrt{d})} \\ w | v}} \legendre{-3}{w}\right) = \prod_{\substack{v \in M_K \\ v \nmid 6\infty \\ \ord_v(d) \equiv 1, 5 \pmod{6}}} \legendre{-3}{v}.
\]
Hence
\begin{align}
\label{eq_quadratic}
&\rk(E_d/K(\sqrt{d}))-\rk(E/K(\sqrt{d})) \equiv \sum_{v | 6\infty} \eta_v(d) \pmod{2} \nonumber \\
&\iff \prod_{\substack{v \in M_K \\ v \nmid 6\infty \\ \ord_v(d) \equiv 1, 5 \pmod{6}}} \legendre{-3}{v} = 1.
\end{align}

Finally, combining (\ref{eq_4th_powers}), (\ref{eq_3rd_powers}), and (\ref{eq_quadratic}) with (\ref{eq_all_twists}) proves the claim, with
\[
\epsilon_v(d) := \eta_v(d) - \delta_v(A, d^2) + \omega_v(d).
\]
\end{proof}

As in the quartic twist case, we have the following useful fact.

\begin{prop}\label{mod_6_prop}
Let $v | 6$ be a place of $K$, and let $c, d \in K_v^*$.  Define the integer $m_v$ by
\[
m_v := \begin{cases} 2e_{v/2} + 1 &\mbox{if $v | 2$} \\ \lceil \frac{3}{2}e_{v/3} \rceil + 1 &\mbox{if $v | 3$,} \end{cases}
\]
where $e_{v/p}$ is the ramification index of $v$ over $p$.  Suppose that
\begin{itemize}
  \item $\ord_v(c) \equiv \ord_v(d) \pmod{6}$
  \item $c/(\pi_v^{\ord_v(c)})$ and $d/(\pi_v^{\ord_v(d)})$ have the same residues in $(\Or_v/v^{m_v})^*/((\Or_v/v^{m_v})^*)^6$.
\end{itemize}
Then $c \equiv d \pmod{(K_v^*)^6}$.
\end{prop}
\begin{proof}
Without loss of generality, $c$ and $d$ are units in $\mc{O}_v$ which have the same residues in $(\Or_v/v^{m_v})^*/((\Or_v/v^{m_v})^*)^6$.  When $v | 2$ (respectively, $v | 3$), the assumption that $c$ and $d$ have the same residues in $(\Or_v/v)^* / ((\Or_v/v)^*)^3$ (respectively, $(\Or_v/v)^* / ((\Or_v/v)^*)^2$) implies that $c \equiv d \pmod{(K_v^*)^3}$ (respectively, $(K_v^*)^2$) by Hensel's Lemma.  The corresponding claim modulo $(K_v^*)^2$ (respectively, $(K_v^*)^3$) follows by \cite[Corollary 4.3]{quadratic_twists}, except that we have added a ceiling function in $m_v$ when $v | 3$ because we do not assume $\mu_3 \subset K$.
\end{proof}

\section{Example: $K = \mb{Q}$}\label{example_q}
The explicit conditions on $c$ and $d$ given in Propositions \ref{mod_4_prop} and \ref{mod_6_prop} allow us to easily classify the conjectural rank parities of quartic and sextic twists using existing computational tools.  To illustrate this, SAGE \cite{sage, pari} was used to classify the conjectural rank parities of all twists of $y^2 = x^3 + x$ and $y^2 = x^3 + 1$ over $\Q$.

For the quartic twists of $y^2 = x^3 + x$, we have $m_2 = 4$.  According to SAGE, we have the representatives
\[
(\Or_2/2^{4})^*/((\Or_2/2^{4})^*)^4 = \{1, 3, 5, 7, 9, 11, 13, 15\}.
\]
Then using Simon's two-descent algorithm, we get the local invariants shown in Fig.\ 1.  Note that these results agree with \cite[Proposition X.6.2]{silverman}.

\begin{figure}[h]
\centering
\fbox{
\begin{tabular}{r|r||c|c|c|c|c|c|c|c}
$\delta_2(A, d)$ & Unit at 2 & 1 & 3 & 5 & 7 & 9 & 11 & 13 & 15 \\ \hline
$\ord_2$ &&&&&&&&& \\ \hline \hline
0&& 0&1&1&0&0&0&1&1 \\ \hline
1&& 0&0&0&0&0&0&0&0 \\ \hline
2&& 0&0&1&1&1&0&0&1 \\ \hline
3&& 1&1&1&1&1&1&1&1
\end{tabular}
}
\fbox{
\begin{tabular}{r||c|c}
Sign & $+1$ & $-1$ \\ \hline
$\delta_\infty(A, d)$ & 0 & 1
\end{tabular}
}
\caption{Local invariants $\delta_v(A, d)$ for quartic twists of $y^2 = x^3 + x$ over $\Q$.  Each cell in the $\delta_2(A, d)$ table corresponds to one class in Proposition \ref{mod_4_prop}.  Due to the method used to compute these values, they are conditional on the Shafarevich-Tate conjecture.  Note $\rk(y^2 = x^3 + x / \Q) = 0$.\vspace{20pt}}
\end{figure}

For the sextic twists of $y^2 = x^3 + 1$, we have $m_2 = 3$ and $m_3 = 3$.  According to SAGE, we have the representatives
\[
(\Or_2/2^{3})^*/((\Or_2/2^{3})^*)^6 = \{1, 3, 5, 7\}
\]
\[
(\Or_3/3^{3})^*/((\Or_3/3^{3})^*)^6 = \{1, 2, 4, 5, 8, 16\}.
\]
Then using Simon's two-descent algorithm, we get the local invariants shown in Fig.\ 2.

\begin{figure}[h]
\centering
\fbox{
\begin{tabular}{r|r||c|c|c|c}
$\epsilon_2(d)$ & Unit at 2 & 1 & 3 & 5 & 7 \\ \hline
$\ord_2$ &&&&& \\ \hline \hline
0&& 0&1&0&1 \\ \hline
1&& 1&1&1&1 \\ \hline
2&& 0&1&0&1 \\ \hline
3&& 1&1&1&1 \\ \hline
4&& 1&1&1&1 \\ \hline
5&& 1&1&1&1
\end{tabular}
}
\fbox{
\begin{tabular}{r|r||c|c|c|c|c|c}
$\epsilon_3(d)$ & Unit at 3 & 1 & 2 & 4 & 5 & 8 & 16 \\ \hline
$\ord_3$ &&&&&&& \\ \hline \hline
0&& 0&0&0&1&0&1 \\ \hline
1&& 0&1&0&1&1&0 \\ \hline
2&& 1&0&1&0&0&1 \\ \hline
3&& 1&0&0&1&1&1 \\ \hline
4&& 0&1&0&1&1&0 \\ \hline
5&& 1&0&1&0&0&1
\end{tabular}
}

\fbox{
\begin{tabular}{r||c|c}
Sign & $+1$ & $-1$ \\ \hline
$\epsilon_\infty(d)$ & 0 & 1
\end{tabular}
}
\caption{Local invariants $\epsilon_v(d)$ for sextic twists of $y^2 = x^3 + 1$ over $\Q$.  Each cell in the $\epsilon_2(d)$ or $\epsilon_3(d)$ table corresponds to one class in Proposition \ref{mod_6_prop}.  Due to the method used to compute these values, they are conditional on the Shafarevich-Tate conjecture.  Note $\rk(y^2 = x^3 + 1 / \Q) = 0$.}
\end{figure}

% \begin{table}[p]
% \includegraphics[width=\textwidth]{quartic_and_sextic_twists__large_table_cropped.pdf}
% \caption{Conjectural rank parities for sextic twists of $y^2 = x^3 + 1$ by positive $d$ over $\Q$.  Each cell corresponds to one positive class in Corollary \ref{main_cor2}.  A cell's color indicates the rank parity for a positive representative $d$ of that class satisfying $\legendre{-1}{U(d)} = 1$.  For negative $d$, all values are flipped.  {\color[gray]{0.2} $\blacksquare$}$= 0$, {\color[gray]{0.8} $\blacksquare$}$= 1$.}
% \end{table}

\noindent{\bf Acknowledgments:} The author would like to thank Majid Hadian for providing mentorship on this project, as well as for his collaboration on a previous paper \cite{quadratic_twists} which inspired this one.  The author was partially supported by a Summer Undergraduate Research Fellowship from the California Institute of Technology's Student-Faculty Programs Office.

% \bibliography{quartic_and_sextic_twists_bib}{}
% \bibliographystyle{plain}

\end{document}